\documentclass[12pt]{amsart}

\usepackage{amssymb, amscd}
\usepackage{latexsym,epsfig}
\usepackage[all]{xy}
\usepackage[pdf]{pstricks}
\usepackage{hyperref}

\numberwithin{equation}{section}
\def\today{\ifcase\month\or Jan\or Febr\or  Mar\or  Apr\or May\or Jun\or  Jul\or Aug\or  Sep\or  Oct\or Nov\or  Dec\or\fi \space\number\day, \number\year}

\newcommand{\CC}{\mathbb C}
\newcommand{\EE}{\mathbb E}

\newcommand{\QQ}{\mathbb Q}
\newcommand{\RR}{\mathbb R}

\newcommand{\ZZ}{\mathbb Z}

\newcommand{\Sym}{{\mathrm{Sym}}}

\numberwithin{equation}{section}
\newtheorem{theorem}{Theorem}[section]
\newtheorem{observation}[theorem]{Observation}

\newtheorem{proposition}[theorem]{Proposition}

\newtheorem{definition-lemma}[theorem]{Definition-Lemma}
\theoremstyle{definition}

\theoremstyle{remark}
\newtheorem{remark}[theorem]{Remark}

\begin{document}

\title[Modular Forms via Invariant Theory]{Modular Forms via Invariant Theory}
\begin{abstract}
We discuss two simple but useful observations that allow the construction
of modular forms from  given ones using invariant theory. The first one 
deals with elliptic modular forms and their derivatives, and generalizes
the Rankin-Cohen bracket, while the
second one deals with vector-valued modular forms of genus greater than $1$.
\end{abstract}

\author{Fabien Cl\'ery}
\address{Institute of Computational and Experimental
Research in Mathematics,
121 South Main Street, Providence, RI 02903, USA}
\email{cleryfabien@gmail.com}

\author{Gerard van der Geer}
\address{Korteweg-de Vries Instituut, Universiteit van
Amsterdam, Postbus 94248,
1090 GE  Amsterdam, The Netherlands}
\email{g.b.m.vandergeer@uva.nl}

\maketitle

\begin{section}{Introduction}\label{sec-intro}
In this paper we present two observations on the use of invariant theory
in the theory of modular forms. 

The first observation is that the invariant theory of binary forms
of degree $r$ can be applied to an elliptic modular form to
produce in a very simple and direct way 
new modular forms that are expressions in the first $r$
derivatives of the elliptic modular form.
For example, the $r$th transvectant
of a binary form of degree $r$ produces a Rankin-Cohen bracket.
But one can use all invariants and there are many more invariants than
just the transvectants that give rise to the Rankin-Cohen brackets.
The novelty is to associate a binary form to a modular form and a natural
number and to use invariant theory of this binary form.
It extends to multi-invariants of several binary forms and 
then associates a modular form to a tuple of elliptic modular forms and
their derivatives.

The second observation deals with vector-valued modular forms.
We can view a modular form as a section of an automorphic 
vector bundle on a moduli space
or on an arithmetic quotient of a bounded symmetric domain. 
Such a bundle corresponds to a factor of automorphy, or equivalenly,
it is obtained by applying a Schur functor to the Hodge bundle.
In turn, this is given by  a representation of a group, usually a general
linear group, and
the weight of the modular form 
refers to this representation.
For example, for a Siegel modular form of degree
$g$ the corresponding factor of automorphy 
is given by an irreducible representation  
of ${\rm GL}(g)$. 

The second observation of this paper is now 
that we can apply invariant
theory for this  ${\rm GL}(g)$-representation  to construct 
new modular forms from a given one. 

For example, for Siegel modular forms of degree $2$ the factors 
of automorphy are indexed by the irreducible representations of
${\rm GL}(2)$ and here we can apply the invariant theory of binary
forms to create new modular forms from a given one.

When we apply this observation 
to the invariant theory of binary sextics 
and a non-zero 
Siegel modular cusp form $\chi_{6,8}$ of degree $2$ and
weight $(6,8)$, we recover the method of 
\cite{CFG1} to construct
all degree two Siegel modular forms on ${\rm Sp}(4,{\ZZ})$ 
from just the form~$\chi_{6,8}$.  

Similarly, for degree $3$ we
apply it to the invariant theory of ternary quartics and recover
the method of the paper \cite{CFG2} that allows one to construct 
all Siegel modular forms on ${\rm Sp}(6,{\ZZ})$ from one cusp form $\chi_{4,0,8}$
of weight $(4,0,8)$. 

The second observation  applies equally well to all kinds of arithmetic groups,
to modular
forms on other bounded symmetric domains, or to Teichm\"uller forms
on the moduli spaces of curves.

We illustrate this method by examples involving Siegel modular forms,
Teichm\"uller forms and Picard modular forms.

We finish the paper by giving an example of a generalization of the first
observation to vector-valued modular forms of more variables.

\end{section}
\section*{Acknowledgements}
The first author was supported by Simons Foundation Award
546235 at the Institute for Computational and
Experimental Research in Mathematics at Brown University.
The second author thanks ICERM for hospitality enjoyed during a visit
when this paper was written. We thank Carel Faber and Don Zagier for 
useful remarks. We thank the referee for pointing out inaccuracies in an
earlier version.

\begin{section}{The First Observation}
Here we start with an elliptic modular form $f$
of weight $k$
on some congruence subgroup $\Gamma$ of ${\rm SL}(2,{\ZZ})$. 
The space of modular forms  of weight $k$ on $\Gamma$
is denoted by $M_k(\Gamma)$ and the subspace of cusp forms by $S_k(\Gamma)$.
We write $\tau$ for the variable in
the upper half plane~$\mathfrak{H}$.

We consider the derivatives $f^{(n)}=d^nf/d\tau^n$ for $n=0,\ldots,r$ of $f$
and using these derivatives we shall associate a modular form 
to each invariant of binary forms
of degree~$r$.

Let $V=\langle x_1,x_2 \rangle$ be the vector space generated by $x_1$
and $x_2$. The group ${\rm GL}(V)={\rm GL}(2)$ acts on $V$.
Recall that an invariant of a binary form $\sum_{i=0}^r a_i \binom{r}{i}x_1^{r-i}x_2^i
\in {\rm Sym}^r(V)$ of degree~$r$ is a (homogeneous) 
polynomial in the coefficients 
$a_0,\ldots,a_r$ invariant under ${\rm SL}(V)$. 
By its degree we mean the degree in the $a_i$. An invariant  
has order $n$ if for 
all its monomials the sum of the indices of the $a_i$ is equal to $n$;
equivalently, if it changes under ${\rm GL}(V)$ by the $n$th power
of the determinant.  We note that for an invariant of degree $d$
and order $n$ we have $dr=2n$.
Such an invariant corresponds to a ${\rm GL}(V)$-equivariant embedding 
$\det(V)^{\otimes dr} \hookrightarrow {\rm Sym}^d({\rm Sym}^r(V))$.
We will write $V_r$ for ${\rm Sym}^r(V)$. 
The set of invariants $\mathcal{I}(V_r)$ forms in a natural way a ring graded by
the degree: $\mathcal{I}(V_r)=\oplus_d \mathcal{I}_d(V_r)$.
 
\begin{theorem}\label{thm}
Suppose that $I$ is an invariant of degree $d$  and order $n$ of
a binary form $\sum_{i=0}^r a_i \binom{r}{i} x_1^{r-i}x_2^i$ 
of degree~$r$ and let
$f$ be an elliptic modular form of weight $k$ on a congruence subgroup
$\Gamma$ of ${\rm SL}(2,{\ZZ})$.
Then by the substitution
$$
a_i \mapsto i! \, \binom{k+r-1}{i}\, f^{(r-i)}
$$
for $i=0,\ldots,r$ in $I$,
we obtain a map
$ \Psi: \mathcal{I}_d(V_r) \times M_k(\Gamma)
 \longrightarrow M_{d k+2n}(\Gamma) $.
For fixed~$f$, the map $I\mapsto \Psi(I,f)$ defines a homomorphism
$\mathcal{I}(V_r) \to R(\Gamma)$ with $R(\Gamma)$ the ring of modular forms 
on $\Gamma$.
\end{theorem}
\begin{proof} In the proof we will write $\delta$ for the degree of $I$.
To $f\in M_k(\Gamma)$ we associate the vector-valued function
$$
F=(F_0,\ldots,F_r)^t \colon \mathfrak{H} \to \CC^{r+1}\, ,
$$
where $F_i(\tau)=i! \binom{k+r-1}{i} f^{(r-i)}(\tau)$
for $i=0, \ldots,r$.
By repeatedly differentiating with respect to $\tau$ the functional equation
$f((a\tau+b)/(c\tau+d))=(c\tau +d) ^k f(\tau)$ 
we get
$$
F\left(\frac{a\tau+b}{c\tau+d}\right)=
(c\tau+d)^k\,
\Sym^r
\left(
\begin{matrix}
(c\tau+d)^2 & c(c\tau+d) \\ 0 & 1
\end{matrix}
\right)
F(\tau) \eqno(\star)
$$
for any
$\gamma=\left(
\begin{matrix}
a & b \\ c & d
\end{matrix}
\right) \in \Gamma$.
Thus by the substitution $a_i \mapsto F_i(\tau)$ we obtain a binary
form of degree $r$ and $\gamma$ defines an action on it by  
$$
A=\left( \begin{matrix} (c\tau+d)^2 & c(c\tau+d) \\ 0 & 1 \\ \end{matrix}
\right) \in {\rm GL}(2,{\CC})\, .
$$
Since $I$ is an invariant of order $n$ it follows that $\Psi(I,f)$ 
changes by $\det(A)^n=(c\tau+d)^{2n}$ under this action. As $I$ is of degree
$\delta$  we get under the substitution 
$\tau \mapsto (a\tau+b)/(c\tau+d)$ in $F$
also a factor $(c\tau+d)^{\delta k}$ in $\Psi(I,f)$ as equation ($\star$) shows.
Together we get that $\Psi(I,f)$ transforms as a modular form of 
weight $\delta k+2n$ on $\Gamma$.
The fact that $\Psi(I,f)$ is holomorphic on $\mathfrak{H}$ is clear and 
also the conditions
at the cusps of $\Gamma$ are easily checked. That for
fixed $f$ the map $I \to \Psi(I,f)$ is 
obtained by substitution makes it clear that it is a homomorphism.
\end{proof}
Note that $\Psi(I,cf)=c^d \Psi(I,f)$ for $I$ of degree $d$ and $c \in {\CC}$.
Often we shall use the notation 
$$
\Psi_I(f)=\Psi(I,f)\, .
$$
\begin{remark} Of course,
the theorem applies as well to modular forms with a character. We then get a map
$$
\Psi_I: M_k(\Gamma,\chi) \to M_{kd+2n}(\Gamma,\chi^d)\, ,
$$
where $d$ is the degree of $I$.
\end{remark}
The theorem makes the relation between transvectants and Rankin-Cohen brackets
transparent. Recall that invariant theory associates to a pair $(F,G)$ of binary forms of degree
$m$ and $n$ a so-called transvectant in $V_{m+n-2r}$ defined by
$$
(F,G)_r=\frac{(m-r)!(n-r)!}{m! \, n!} \sum_{j=0}^r (-1)^j \binom{r}{j}
\frac{\partial^r F}{\partial x_1^{r-j}\partial x_2^{j} }
\frac{\partial^r G}{\partial x_1^{j}\partial x_2^{r-j} } \, .
$$
We now replace $r$ by $2r$ and take $m=n=2r$ and $F=G$. 
If $F \in V_{2r}$ denotes the universal binary form of degree $2r$
then the invariant $I=(F,F)_{2r}$ is of degree $2$ and 
order $2r$. We apply this to the binary form $F$ associated 
as in the proof of Theorem \ref{thm} 
to a modular form $f \in M_k(\Gamma)$ and 
it gives
$$
\Psi_I(f)= (2r)! \, (2\pi i)^{2r} \,[f,f]_{2r} \in M_{2k+4r}(\Gamma)
$$
with $[f,f]_{2r}$ the Rankin-Cohen bracket. Recall that for a pair of
modular forms $f_1 \in M_{k_1}(\Gamma)$, $f_2 \in M_{k_2}(\Gamma)$ 
the $r$th Rankin-Cohen bracket $[f,g]_{r}$
is defined as
\[
[f,g]_r=\frac{1}{(2\pi i)^r}
\sum_{n+m=r}
(-1)^n
\binom{k_1+r-1}{m}
\binom{k_2+r-1}{n}
\frac{d^n f}{d\tau^n}
\frac{d^m g}{d\tau^m}
\]
and is an element of $M_{k_1+k_2+2r}(\Gamma)$ and is a cusp form for $r>0$.
It was introduced in \cite{Cohen} using results of \cite{Rankin}.
\begin{remark}
Note that for the case $r=0$ we are dealing with the covariant $FG$ and the product $fg$.
\end{remark}

We give an example.
The ring of invariants of $V_3$ is generated by 
$$
I_3=a_0^2a_3^2-6\,a_0a_1a_2a_3+4\,a_0a_2^3+4\,a_1^3a_3-3\,a_1^2a_2^2\, .
$$
By calculating the first few Fourier coefficients
we see that 
it gives for the normalized Eisenstein series $E_k=1 - (2k/B_k) \sum_{n\geq 1} \sigma_{k-1}(n)q^n$
on ${\rm SL}(2,{\ZZ})$ of weights $4$ and $6$ the results
$$
\Psi_{I_3}(E_4)=-53084160000\pi^6\,E_4\Delta^2, \quad
\Psi_{I_3}(E_6)=-203928109056\pi^6\,(8\, E_4^3+E_6^2)\Delta^2\, .
$$
\subsection{Examples from multi-invariants}
As the above suggests we can also apply invariant theory for a tuple of
binary forms. 
The invariant theory for the diagonal action of
${\rm GL}(2)$ on a direct sum
$V_{m_1}\oplus \cdots \oplus V_{m_s}$ provides a plethora
of invariants.

We may take the invariant $I=(F,G)_r$ for $F,G \in V_r$
and apply it to a pair of modular forms $(f,g)$. By
associating a binary form $F$ to $f$ and a binary form  
$G$ to  $g$ we can apply
multi-invariants and we find by applying the
substitution of Theorem \ref{thm} to both pairs $(F,f)$ and $(G,g)$ that
$$
\Psi_I(f,g)= (-1)^r r! (2\pi i)^r   [f,g]_r \, .
$$
We recover in a transparent way
the relation between Rankin-Cohen brackets and transvectants. 
This relation was apparently first observed by Zagier \cite[p.\ 74]{Zagier1994},
and there is an extensive literature on this relation, 
see for example~\cite{Olver}.

\smallskip
But there are many more invariants, bi-invariants and multi-invariants
than the trans\-vectants.
To a $s$-tuple of modular forms $(f_1,\ldots,f_s)$ 
with $f_j\in M_{k_j}(\Gamma)$ we associate a $s$-tuple 
$(F_1,\ldots,F_s)$ of binary forms $F_j \in V_{r_j}$
as in the proof of Theorem \ref{thm}. 
An invariant $I$ of the action of ${\rm GL}(2)$ on $\oplus V_{r_j}$
of degree $d_j$ in the coefficients of the
binary form $F_{j}$ 
and of order $n$ defines a map
$$
\Psi_I: \oplus_{j=1}^s M_{k_j}(\Gamma) 
\longrightarrow M_{2n+\sum_{j=1}^s d_jk_j}(\Gamma)\, 
$$
by substituting $i! \binom{k_j+r_j-1}{i}f_j^{(r_j-i)}$ for
the coefficient $a_i^{(j)}$ of the binary form $F_j$.

As a concrete example,
we take $(F,G)\in V_3\oplus V_1$ with binary forms 
$F=\sum_{i=0}^3 a_i\binom{3}{i} x_1^{3-i}x_2^i$
and $G=b_0x_1+b_1x_0$. In this case the  generators of the
ring of invariants are known, see \cite{Draisma}. 
For example, there is an invariant
$$
I= (F,G^3)_3=a_0b_1^3-3\, a_1b_0b_1^2+ 3\, a_2 b_0^2b_1-a_3b_0^3 \, .
$$
It defines a map
$\Psi_I: M_{k_1}(\Gamma) \times M_{k_2}(\Gamma) \longrightarrow M_{k_1+3k_2+6}(\Gamma)$,
e.g., 
$$
\Psi_{I}(E_6,E_4)=\sqrt{-1} \, 86016\pi^3\,\Delta(E_4^3+2\, E_6^2).
$$
As another example, take binary forms $F,G,H$ of degrees
$3,2,1$ with coefficients $a_i,b_i,c_i$ and consider the tri-invariant
$$
I=a_0b_2c_1-2\,a_1b_1c_1-a_1b_2c_0+a_2b_0c_1+2\,a_2b_1c_0-a_3b_0c_0\, .
$$
For $f,g,h$ modular forms of weights $k_1,k_2,k_3$ on $\Gamma$
we have $\Psi_I(f,g,h)\in M_{k_1+k_2+k_3+6}(\Gamma)$. 

\end{section}

\begin{section}{The Second Observation}
Here we will be dealing with vector-valued modular forms and not
with derivatives.
We formulate the second observation for the case where the modular
form has weight~$\rho$, with $\rho$ denoting the highest weight
of an irreducible representation 
$V$ of ${\rm GL}(g)$. Let $U$ be the standard representation 
of ${\rm GL}(g)$.
The modular form can live on a bounded symmetric 
domain with factor of automorphy $\rho$, or on a moduli space where
it is a section of the vector bundle ${\EE}_{\rho}$ constructed
from the Hodge bundle ${\EE}$ by a Schur functor defined by $\rho$.

Recall that an invariant relative to the action of ${\rm GL}(g)$ on the
irreducible representation $V$ of highest weight $\rho$ is an element in
the algebra on $V$ invariant under ${\rm SL}(g)$. 
In fact, an invariant can be obtained by an equivariant embedding of
${\rm GL}(g)$-representations
$$
\det(U)^{\otimes m} \hookrightarrow {\rm Sym}^d(V).  \eqno(1)
$$
Viewing
the fibre of the dual bundle ${\EE}_{\rho}^{\vee}$ of ${\EE}_{\rho}$
as a ${\rm GL}(g)$-representation
and a section $f$ 
of ${\EE}_{\rho}$ as a function on ${\EE}_{\rho}^{\vee}$,
we can evaluate  ${\rm Sym}^d(f)$ via the embedding (1) on
the subbundle $(\det{\EE}^{\vee})^{\otimes m}$ and
obtain a section of the bundle $\det({\EE})^m$. 

More generally, if $W$ is an irreducible 
representation of ${\rm GL}(g)$ of highest weight $\sigma$,
an equivariant embedding 
$$
W \hookrightarrow {\rm Sym}^d(V) \eqno(2)
$$
defines a concomitant of type $(d,\sigma)$ for $\rho$.
Equivalently, the equivariant embedding $\phi: 
W \hookrightarrow {\rm Sym}^d(V)$ may be viewed  as an equivariant
embedding
$\phi':{\CC} \to {\rm Sym}^d(V)\otimes W^{\vee}$. Then the image
$\phi'(1)$ is called a concomitant.

Again, if $f$ is a modular form, say a section of ${\EE}_{\rho}$,
then we can view $f$ as a function on ${\EE}_{\rho}^{\vee}$, and
${\rm Sym}^d(f)$ as a function on ${\rm Sym}^d({\EE}_{\rho}^{\vee})$.
Then by restriction to $({\EE}^{\vee})_{\sigma}$ via (2) 
we obtain a modular form of weight $\sigma$
from $f$. Alternatively, by projection in ${\rm Sym}^d({\EE}_{\rho})$
of the section ${\rm Sym}^d(f)$ we get a form of weight $\sigma$. 
By the phrase `applying the invariant or concomitant
to~$f$' we mean this evaluation. We thus obtain the following.

\begin{observation}\label{observation}
Let $f$ be a modular form of weight $\rho$ and let $I$ be a concomitant
of type $(d,\sigma)$ for some $d$ 
for the action of ${\rm GL}(g)$ on the representation $\rho$. 
Then by applying $I$ to $f$ one obtains a modular form of weight $\sigma$. 
\end{observation}

Instead of spelling this out in detailed notations we will illustrate
it by a couple of examples in the next section.
One may formulate variants of this involving a finite set of 
modular forms to which invariant theory (via multi-invariants) is applied.
\end{section}
\begin{section}{Illustrations}
We now illustrate the second observation by a number of special cases.

\subsection{Siegel modular forms of degree two}
Here we are dealing with the representation theory of ${\rm GL}(2)$ 
and the invariant theory of binary forms. 
There is a wealth of explicit results in invariant theory
that can be applied.

Let $\Gamma \subset {\rm Sp}(4,{\QQ})$ be a group commensurable with
${\rm Sp}(4,{\ZZ})$. We let $\rho$ be an irreducible representation of
${\rm GL}(2)$ of highest weight $(j+k,k)$.
By Observation \ref{observation} we find for a given modular form 
$f \in M_{j,k}(\Gamma)$, that is, 
a section of ${\rm Sym}^j({\EE})\otimes \det({\EE})^k$,  
a homomorphism 
$$
\Psi_{\bullet}(f): \mathcal{I}(U_j) \longrightarrow R(\Gamma), \qquad
J \mapsto \Psi_J(f) \, ,
$$
of the ring of invariants $\mathcal{I}(U_j)$ of binary forms of degree $j$
to the ring $R(\Gamma)$ of scalar-valued modular forms on $\Gamma$.

Let us illustrate this with the simplest non-trivial case for $j=2$.
The discriminant $I=a_1^2-4a_0a_2$ of a quadratic form 
$a_0x_1^2+a_1x_1x_2+a_2x_2^2$ gives an invariant of degree $2$.
If $f$, the transpose of $(f_0,f_1,f_2)$, is a modular form of weight
$(2,k)$ on some congruence subgroup $\Gamma$ of
${\rm Sp}(4,{\ZZ})$ then we find a scalar-valued one
$\Psi_I(f)=f_1^2-4f_0f_2 \in M_{0,2k+2}(\Gamma)$.
The weight $(0,2k+2)$ follows from the identity of ${\rm GL}(2)$-representations 
$$
{\rm Sym}^2(V_2)=V_4\oplus\det(V)^2\, .
$$

This can be extended to tuples of modular forms by using
multi-invariants of binary forms. We write $U_n={\rm Sym}^n(U)$
with $U$ the standard representation of ${\rm GL}(2)$.
We consider the ring of invariants
$$
\mathcal{I}_{n_1,\ldots,n_m}=\mathcal{I}(U_{n_1} \oplus
\cdots \oplus U_{n_m})
$$
relative to the action of ${\rm GL}(2)$. One can consider 
elements of $\mathcal{I}_{n_1,\ldots,n_m}$ as (multi-)invariants of a $m$-tuple
of binary forms $b_1,\ldots,b_m$. 
Let  $J \in \mathcal{I}_{n_1,\ldots,n_m}$ be a multi-invariant that is
of degree $(d_1,\ldots,d_m)$ in the coefficients of
the binary forms $b_1,\ldots,b_m$. Then applying $J$ to an
$m$-tuple of modular forms $f_i \in M_{j_i,k_i}(\Gamma_i)$
defines a map
$$
\prod_{i=1}^m M_{j_i,k_i}(\Gamma) \longrightarrow
M_{0,k}(\Gamma) 
\qquad \text{ with $k=\sum_{i=1}^m d_i(k_i+j_i/2)$.}
$$
The weights are obtained from representation theory as in the example above. 
For example, take $(n_1,n_2)=(4,2)$. If we write
$$
b_1= \sum_{i=0}^4 \alpha_i x_1^{4-i}x_2^i, \quad
b_2=\sum_{i=0}^2 \beta_i x_1^{2-i}x_2^i\, , 
$$
where now (unlike before) 
for convenience we do not use binomial coefficients in the binary forms,
we have the invariant
$$
J_{1,2}= 6\, \alpha_0\beta_2^2 -3\, \alpha_1\beta_1 \beta_2 
+ 2\, \alpha_2 \beta_0 \beta_2 + \alpha_2\beta_1^2 -3\, \alpha_3\beta_0\beta_1 + 6\, \alpha_4 \beta_0^2\, .
$$ 
The invariant $J_{1,2}$ defines a map
$$
M_{4,k_1}(\Gamma_1) \otimes M_{2,k_2}(\Gamma_2) 
\longrightarrow 
M_{0,k_1+2k_2+4}(\Gamma)\, .
$$

Next we look at covariants of binary forms. Recall that the ring 
of covariants $\mathcal{C}(U_j)$ 
of a binary form of degree $j$ can be identified 
with the ring of invariants $\mathcal{I}(U_j\oplus U_1)$ 
of $U_j \oplus U_1$, see \cite[3.3.9]{Springer}.  If we write
an element of $U_1$ as $l_1x_1+l_2x_2$, the isomorphism 
can be given explicitly by substituting $l_1=-x_2$ and $l_2=x_1$
in an invariant of $U_j\oplus U_1$.

The following proposition is a direct consequence of Observation 
\ref{observation}.
\begin{proposition}
If $C \in \mathcal{C}(U_j)$ is a covariant of 
degree $a$ in the coefficients
of the binary form and of degree $b$ in $x_1,x_2$,
then applying $C$ defines a map
$$
\Psi_C: M_{j,k}(\Gamma) \longrightarrow
M_{b,a(k+j/2)-b/2}(\Gamma)\, .
$$
\end{proposition} 
\begin{proof} We write $U[n+m,m]$ for ${\rm Sym}^n(U)\otimes \det(U)^m$.
The covariant $C$ corresponds to an equivariant embedding 
$U[b+l,l]\hookrightarrow {\rm Sym}^a(U[j+k,k])$ for some $l$.
To determine $l$ we take out a factor $\det(U)^{\otimes k}$ 
and look at the 
irreducible representations occurring in
${\rm Sym}^a(U[j,0])$ and these are of the form 
$U[aj-r,r]$ for non-negative $r$.
Therefore, if $U[b+l,l]$ occurs then $(b+l,l)=(aj-r+ak,r+ak)$, that is, $2l=aj+2ak-b$.
\end{proof}
If we fix a modular form $f \in M_{j,k}(\Gamma)$ we get
an induced map
$$
\Psi_{\bullet}(f): \mathcal{C}(U_{j}) \to M(\Gamma), \quad C \mapsto \Psi_C(f)\, ,
$$
where $M(\Gamma)$ is the ring of vector-valued Siegel modular forms
on $\Gamma$ of degree $2$.

In the paper \cite{CFG1} it was shown that for 
$\Gamma={\rm Sp}(4,{\ZZ})$, $j=6$ and $f=\chi_{6,8}$, 
a generator of the space of cusp forms
$S_{6,8}({\rm Sp}(4,{\ZZ}))$, every
vector-valued modular form on ${\rm Sp}(4,{\ZZ})$ can be obtained
from a form $\Psi_C(\chi_{6,8})$ for a $C \in \mathcal{C}(U_6)$
after dividing by an appropriate power of
the cusp form $\chi_{10}$ of weight~$10$.

Alternatively, 
using the meromorphic modular form 
$\chi_{6,-2}=\chi_{6,8}/\chi_{10}$, we found maps
$$
M({\rm Sp}(4,{\ZZ})) \longrightarrow \mathcal{C}(U_{6})
\xrightarrow{\Psi_{\bullet}(\chi_{6,-2})} M({\rm Sp}(4,{\ZZ}))[1/\chi_{10}]\, ,
$$
the composition of which is the identity.

\smallskip

A variation of this deals with multi-covariants.  We can also allow 
modular forms with a character.

\begin{proposition}
If $C$
is a covariant in $\mathcal{C}(U_{j_1} \oplus \cdots \oplus U_{j_m})$
of degree $a_i$ in the coefficients of the binary form in $U_{j_i}$
and degree $b$ in $x_1,x_2$, then $C$ defines a map
$$
\otimes_{i=1}^m M_{j_i,k_i}(\Gamma,\chi_i) \longrightarrow
M_{b,k}(\Gamma, \chi_1^{a_1}\cdots\chi_m^{a_m})
$$
with $k= \sum_{i=1}^m a_i(k_i+j_i/2) -b/2$.
\end{proposition}
\subsection{Siegel and Teichm\"uller modular forms of degree three}
If $U_{\rho}$ is an irreducible representation of ${\rm GL}(3)$
of highest weight $\rho$ 
and $f$ is a Siegel modular form of weight $\rho$ on some group 
$\Gamma \subset {\rm Sp}(6,{\QQ})$ commensurable 
with ${\rm Sp}(6,{\ZZ})$, then we get a homomorphism
$$
\Psi_{\bullet}(f): \mathcal{I}(U_{\rho}) \longrightarrow R(\Gamma)
$$
of the ring $\mathcal{I}(U_{\rho})$ of invariants
to $R(\Gamma)$, the ring of scalar-valued modular forms on $\Gamma$.

We can extend this map. 
For a given  irreducible representation $\sigma$ 
of ${\rm GL}(3)$ we let 
$\mathcal{C}_{\sigma}(U_{\rho})$ be the $\mathcal{I}(U_{\rho})$
-module of covariants obtained from equivariant embeddings of
$U_{\sigma}$ into the symmetric algebra on $U_{\rho}$.
We get for a given form $f \in M_{\rho}(\Gamma)$ 
a map
$$
\Psi_{\bullet}(f): \mathcal{C}_{\sigma}(U_{\rho}) \longrightarrow
 M_{\sigma}(\Gamma), \quad
C \mapsto \Psi_C(f)\, ,
$$
where $M_{\sigma}(\Gamma)$ is the $R(\Gamma)$-module 
$\oplus_k M_{\sigma \otimes {\det}^k}(\Gamma)$.

The moduli space $\overline{\mathcal{M}}_3$ of stable curves of genus $3$
carries a Hodge bundle, 
also denoted by ${\EE}$. Its restriction to $\mathcal{M}_3$  is the pullback of the Hodge bundle on $\mathcal{A}_3$ under the Torelli map.
For given irreducible representation $\rho$ 
of ${\rm GL}(3)$
we have a vector bundle ${\EE}_{\rho}$ obtained by a Schur functor
from ${\EE}$ and $\rho$. Sections of ${\EE}_{\rho}$ on
$\overline{\mathcal{M}}_3$ are called Teichm\"uller forms of degree or genus $3$ and weight
 $\rho$. We have a graded ring of scalar-valued Teichm\"uller forms 
$T_3=\oplus_k H^0(\overline{\mathcal{M}}_3, \det({\EE})^k)$ of genus $3$. 
The ring $T_3$ is a quadratic extension of the
ring of scalar-valued Siegel modular forms
$R({\rm Sp}(6,{\ZZ}))$ by $\chi_9$, with
$\chi_9$ the Teichm\"uller modular cusp form of weight $9$ that vanishes 
simply on the closure of the hyperelliptic locus. 
It was introduced by Ichikawa, see \cite{Ichikawa}. Its square is a Siegel
cusp form of weight $18$, the product of the $36$ even theta constants.

Given a Teichm\"uller modular form $f$ of weight $\rho$ we have a similar map
$$
\Psi_{\bullet}(f): \mathcal{C}_{\sigma}(U_{\rho})\longrightarrow T_{\sigma}(\Gamma)\, ,
$$
where $T_{\sigma}(\Gamma)$ is the $T_3$-module 
$$ 
\oplus_k H^0(\overline{\mathcal{M}}_3, {\EE}_{\sigma}\otimes \det{\EE}^{\otimes k})\, .
$$
With $\chi_{4,0,8}$ a generator of the space of cusp forms
$S_{4,0,8}(\rm Sp(6,{\ZZ}))$,
the quotient
$\chi_{4,0,-1}=\chi_{4,0,8}/\chi_9$ is a meromorphic section of 
${\rm Sym}^4({\EE})\otimes \det({\EE})^{-1}$ on $\overline{\mathcal{M}}_3$. 
In \cite{CFG2} we used
this form to construct maps
$$
H^0(\overline{\mathcal{M}}_3, {\EE}_{\sigma})
\longrightarrow \mathcal{C}_{\sigma}({\rm Sym}^4(U))  
\xrightarrow{\Psi_{\bullet}(\chi_{4,0,-1})}
H^0(\overline{\mathcal{M}}_3, {\EE}_{\sigma})[1/\chi_9]
$$
the composition of which is the identity. Here $U$ is the standard representation
of ${\rm GL}(3)$. This enables one to construct
all Teichm\"uller and all Siegel modular forms of genus $3$ on ${\rm Sp}(6,{\ZZ})$ 
by concomitants
for the action of ${\rm GL}(3)$ on ternary quartics using $\chi_{4,0,-1}$.

\subsection{Teichm\"uller forms of genus $3$ and $4$}
In \cite{vdG-K} a Teichm\"uller modular form $f$ of
weight $(2,0,0,8)$, a section of ${\rm Sym}^2({\EE}) \otimes \det({\EE})^8$
on $\overline{\mathcal{M}}_4$,  is constructed. 
Here ${\EE}$ is the Hodge bundle on $\overline{\mathcal{M}}_4$.
This Teichm\"uller modular
form can not be obtained by pulling back a Siegel modular form
under the Torelli morphism. It is associated to the quadric containing
the canonical image of the generic curve of genus $4$.

As an invariant we now take
the discriminant $I$ of a quadratic form in four variables and apply
it to $f$. It yields a scalar-valued Teichm\"uller modular form 
of weight $34$ that vanishes on the closure of the 
locus of non-hyperelliptic curves of genus $4$ 
for which the unique quadric that contains the canonical image
is singular. Its square is the pull back of a Siegel modular form of
degree $4$
and weight $68$ that is the product of all the even theta characteristics.

An analogous case is given by the section $\chi_{2,0,4}$ of ${\rm Sym}^2({\EE}) \otimes 
\det({\EE})^4$ on the Hurwitz space 
$\overline{\mathcal{H}}_{3,2}$ of admissible covers of genus $3$ 
and degree $2$, constructed in \cite{vdG-K} and ${\EE}$ the corresponding Hodge bundle. 
Its discriminant is a modular form 
of weight $14$, related to the discriminant of binary octics, whose
square is the pullback of a Siegel modular cusp form of weight $28$.
\subsection{Picard modular forms}
Here we fix an imaginary quadratic field $F$ with ring of integers $O_F$
and consider a non-degenerate Hermitian form $h=z_1z_2'+z_1'z_2+z_3z_3'$
on the vector space $F^3$ with the prime denoting complex conjugation. 
The group of similitudes of $h$ is an algebraic group $G$
over ${\QQ}$ of type ${\rm GU}(2,1)$. The connected component 
$G^{+}({\RR})$ of $G({\RR})$ acts on the set
$\mathfrak{B}$ of negative complex lines
$$
\mathfrak{B}=\{ L : L \subset F^3\otimes_{\QQ} {\RR},\, \dim_{\CC}L=1, h_{|L}<0\}
$$
which can be identified with the complex $2$-ball 
$\{ (u,v) \in {\CC}^2: v+\bar{v}+u\bar{u}<0\}$. If $\Gamma$ is an arithmetic subgroup of $G$,
the quotient $\Gamma\backslash \mathfrak{B}$ is a moduli space of $3$-dimensional
abelian varieties with multiplication by $F$. It carries a Hodge bundle ${\EE}$
that splits as $W\oplus L$ of rank $2$ and $1$ and we obtain two factors of automorphy. 
We have $\det(W)^6=L^6$. 
We then have modular forms that can be seen as sections of ${\rm Sym}^j(W)
\otimes L^k$. 
If $\det(W)\neq L$ we have to deal with modular forms with a character.

In the paper \cite{CvdG2021} we considered the case where $F={\QQ}(\sqrt{-3})$
and $\Gamma=\Gamma[\sqrt{-3}]$ is a certain congruence subgroup. 
We constructed modular forms 
$\chi_{1,1} \in M_{1,1}(\Gamma[\sqrt{-3}], \det)$ 
and $\chi_{4,4}\in M_{4,4}(\Gamma[\sqrt{-3}],\det^2)$. 
We also have a scalar-valued cusp form $\zeta \in S_{0,6}(\Gamma[\sqrt{-3}],\det)$.
We refer to \cite{CvdG2021} for the notation.
The quotient $\chi_{4,-2}=\chi_{4,4}/\zeta$ is a 
meromorphic modular form of weight $(4,-2)$.
We can apply the second observation to this situation
starting from the two forms $\chi_{1,1}$ and $\chi_{4,-2}$.
We constructed in \cite{CvdG2021}  maps
$$
M(\Gamma) \to \mathcal{C}(V_1 \oplus V_4) \xrightarrow{\Psi_{\bullet}(\chi_{1,1},\chi_{4,-2})} M(\Gamma)[1/\zeta]
$$
from the ring $M(\Gamma)$ of vector-valued modular
forms to the ring $\mathcal{C}(V_1\oplus V_4)$ of bi-covariants 
for the action of ${\rm GL}(2)$ on
$V_1 \oplus V_4$. The map $\Psi_{\bullet}(\chi_{1,1},\chi_{4,-2})$ illustrates
Observation \ref{observation}. The composition of the maps is the identity
and this shows that we can obtain all modular forms this way.

\end{section}
\begin{section}{An example involving both observations}
The example we now treat deals with Picard modular forms living on the $2$-ball
and is inspired by both observations and constructs 
new vector-valued Picard modular
forms from given Picard modular forms and their derivatives. 
For simplicity's sake we deal with modular forms of weight $(1,k)$ 
on the Picard modular group $\Gamma[\sqrt{-3}]$, see \cite{CvdG2021} for 
definitions and notation.
Let $f \in M_{1,k}(\Gamma[\sqrt{-3}])$ be a modular form. 
We write 
$f$ as the transpose of 
$(f_0, f_1)$.
We can view
$f$ as defining a linear form $l=f_0x_1+f_1x_2$.
We put
$\partial f$ as the transpose of
$$
(f_{0u},(f_{1u}+f_{0v})/2,f_{1v})\, ,
$$
where $f_{iu}=\partial f_i/\partial u$ and $f_{iv}=\partial f_i/\partial v$
for $i=0,1$.
We view $\partial f$ as the vector of coefficients of a  binary form $q$ 
of degree $2$.

Writing $V=\langle x_1,x_2 \rangle$ and 
$l=a_0x_1+a_1x_2$ and $q=b_0x_1^2+2\, b_1x_1x_2+b_2x_2^2$ we
consider now multi-invariants for the action of 
${\rm GL}(2)$ on $V_1\oplus V_2$.

\begin{proposition}
Let $I$ be a multi-invariant of the binary forms $l$ and $q$. Assume
that $I$  has degree $d_1$
in the $a_i$ and degree $d_2$ in the $b_i$ and has order $n$. 
Then for $f\in M_{1,k}(\Gamma)$ the expression $\Psi_I(f,\partial f)$
obtained by evaluating the multi-invariant $I$ 
defines a Picard modular form of weight 
$(0,d_1\, k+d_2(k+1)+n)$ on $\Gamma$.
\end{proposition}
The proof follows the pattern of the proof of Theorem \ref{thm}. 
Is seems complicated to
extend it to the case where higher derivatives are involved.

We finish by giving an example. If we let $I$ be the bi-invariant
$$
I=a_0^2b_2-2a_0a_1b_1+a_1^2b_0 
$$
and $f=E_{1,1}$, a vector-valued Eisenstein series defined in \cite{CvdG2021}, 
then it turns out that $\Psi_I(f)$ is a multiple of the 
scalar-valued modular form $\zeta$ of weight $6$ that occured above. 
More
precisely, with $f=E_{1,1} \in M_{1,1}(\Gamma[\sqrt{-3}],\det)$ 
with Fourier expansion
$$
E_{1,1}(u,v)=
\sum_{\alpha \in O_F}
\left[\begin{smallmatrix} X'(\alpha u) \\ \frac{2\pi}{\sqrt{3}}\bar{\alpha}X(\alpha u) \end{smallmatrix}\right]
q_v^{N(\alpha)}
=
\left[\begin{smallmatrix} X'(0) \\ 0 \end{smallmatrix}\right]+
6\left[\begin{smallmatrix} X'(u) \\ \frac{2\pi}{\sqrt{3}}X(u) \end{smallmatrix}\right]\,q_v +
6\left[\begin{smallmatrix} (XYZ)'(u) \\ \frac{2\pi}{\sqrt{3}}3XYZ(u) \end{smallmatrix}\right]\,q^3_v +
\ldots
$$
with $X,Y,Z$ elliptic modular functions, see \cite{CvdG2021} for the
notation, we get 
$$
\Psi_{I}(E_{1,1})(u,v)=8\pi^2a^2
\bigg(
Xq_v+
\frac{9X}{a^2}
(4(XX''-(X')^2)+a^2YZ)
\bigg)q_v^3+\ldots
$$
with $a=X'(0)$,
but we know that $XX''-(X')^2=-a^2YZ$, so we get
$$
\Psi_{I}(E_{1,1})(u,v)=
8\pi^2a^2
(
Xq_v-27XYZq_v^3+\ldots
)
=8\pi^2a^2 \zeta(u,v). 
$$ 
Here $\zeta \in S_{0,6}(\Gamma[\sqrt{-3}],\det)$ is the form that appeared in the
preceding section.

\bigskip

There are no data sets related to this article. The authors are not aware of any conflicts of interest.
\end{section}

\end{document}